\documentclass{article}
\usepackage{CJKutf8}

\newtheorem{fed}{\textbf{Definition}}[section]
\newtheorem{thm}[fed]{\textbf{Theorem}}
\newtheorem{assumption}[fed]{\textbf{Assumption}}

\newtheorem{lemma}[fed]{\textbf{Lemma}}

\newtheorem{prop}[fed]{\textbf{Proposition}}

\usepackage{amscd,amssymb,bbm,graphicx,epsfig,psfrag,epic,eepic,latexsym,amsmath}
\usepackage{amsmath}
\usepackage[arrow,matrix,curve]{xy}
\usepackage{mathrsfs}
\usepackage[normalem]{ulem}
\usepackage{soul}
\usepackage
{color}

\newcommand{\R}{\mathbb{R}}

\newcommand{\e}{\mathbf{e}}

\begin{document}

\title{Exclusion of Infinite Spin for N-body problem in $\R^{d}$}
\author{Xiang Yu, Lei Zhao}

\maketitle

\begin{abstract}
We show that there is no infinite spin at total collisions for $(-\kappa)$-homogeneous N-body problem in higher dimensional Euclidean space $\R^{d}$, in which $0<\kappa<2$ ($\kappa=1$ the Newtonian case), provided the limiting normalized central configuration is isolated and is of dimension $d$ or $d-1$. In the Newtonian case $\kappa=1$, this extends the work of \cite{MM} to $d \ge 3$ and in the $d=3$ case offers a different approach as compared to the current preprint \cite{PZ}.
\end{abstract}

\tableofcontents

\section{Introduction}
In the study of the Newtonian N-body problem, it is now well-known that when a total collision happens, the configurations formed by the particles after proper normalization tend to the set of central configurations. The rotational symmetry of the problem makes it \emph{a priori} possible for a total collision orbit to rotate infinitely many times before ending up colliding each other. 

When the limiting central configuration is non-degenerate, this is shown by Chazy \cite{Chazy} and Saari \cite{Saari} to be impossible. In the planar N-body problem, this problem was recently studied by \cite{Yu} and by Moeckel-Mongomery \cite{MM} who showed that infinite spin does not happen if the limiting central configuration is isolated. Several extensions has been obtained \cite{GSZ, WY}. 
Very recently, this is extended to the spatial N-body problem by Pinzari-Zgliczynski \cite{PZ} by means of constructing canonical variables in the Hamiltonian formalism to perform reduction by the $SO(3)$-symmetry of the problem. 

In this article, we extend the approach of Moeckel-Mongomery  \cite{MM} to deal with the non-abelian rotational symmetry $SO(d)$-symmetry of N-body problem in $\R^{d}$. Two steps need to be carried out, namely reducing the symmetry of the problem to get a reduced equations of motion, and use these reduced equations of motion together with a Lojasiewicz-type argument established in \cite{MM} to conclude that along a total collision orbit the projection to the rotational part of the system is finite.  In contrast to \cite{PZ}, we perform non-abelian Lagrange-Routh reduction at $0$- angular momentum level. 

Remind that central configurations form dilation-invariant, rotation-invariant families. A central configuration is called normalized if the moment of inertia with respect to the center of mass is $1$. A central configuration is called reduced if it represents an equivalent class of central configurations obtained by rotating a fixed central configuration. To better illustrate the universality of the method of \cite{MM} we state our results for N-body problem in $\R^{d}$ with $(-\kappa)$-homogeneous potential for $0<\kappa<2$. 

\begin{thm}\label{thm: main} In the $(-\kappa)$-homogeneous $N$-body system in $\R^{d}$ ($0<\kappa<2$),  along any solution leading to a total collision, if the corresponding limiting reduced normalized central configuration $s_{0}$  is of dimension $d$ or $d-1$, and is isolated, then there is no infinite spin: Without reducing the rotational symmetry, the normalized configuration tends to a fixed limiting normalized central configuration. 
\end{thm}

When $d=3, \kappa=1$, all collinear central configurations are non-degenerate and therefore it follows from the works of Chazy \cite{Chazy} and Saari \cite{Saari} that also no infinite spin happens. This gives a complete result in $d=3$.

We organize the article as follows: In Section \ref{Sec: Gen} we explain our general settings of the problem as well as the separation of rotation part and dilation part from the corresponding reduced normalized configuration. In Section \ref{Sec: Finite Spin} we show that the reduced equations of motion can be analyzed as in \cite{MM}, which lead to a simple proof of Theorem \ref{thm: main}.

\section{General Setting} \label{Sec: Gen}

\subsection{The problem of N-particles and reduction by center of mass}

We consider the $N$-body problem in $\R^{d}, d \ge 1$, in which the particles carry positive masses $m_{1}, \cdots, m_{N}$, in the center of mass frame. The configuration space
$$\mathcal{Q}:=\{q=(q_{1}, \cdots, q_{N}) \in \R^{d}: \sum m_{i} q_{i}=0, q_{i} \neq q_{j}, \forall i \neq j\}.$$
is equipped with the mass-weighted product
$$\langle q, \tilde{q} \rangle =\sum m_{i} q_{i} \cdot \tilde{q}_{i}$$
in which $\cdot $ denotes the inner product in $\R^{d}$.

For $q \in \mathcal{Q}$ we set its length as $r=\sqrt{\langle q ,q \rangle}=: \|q\| \in \R_{+}$, and its corresponding normalized configuration as 
$$\hat{q} \in \mathcal{N}:=\{\hat{q} \in \mathcal{Q}; \|\hat{q}\|=1\}.$$ 

This $N$-body system is a Lagrangian system, with the Lagrangian 
$$\mathcal{L}=\dfrac{1}{2}\langle \dot{q}, \dot{q} \rangle+U(q). $$
The $-\kappa$-homogeneous force function is assumed to take the form
$$U(q)=\sum \dfrac{m_{i} m_{j}}{\|q_{i}-q_{j}\|^{\kappa}}.$$
In particular 
$$U(q)=U(\hat{q})\, r^{-\kappa}.$$

The group $SO(d)$ acts diagonally on $\mathcal{Q}$ (and on $\mathcal{N}$) by rotating each particle around their center of mass. When $d \ge 3$, this group is non-abelian. The corresponding conserved quantities are the components of the angular momentum bivector 
$$C:=\sum m_{i} q_{i} \wedge \dot{q}_{i}.$$ 

The symmetry reduction procedure is more complicated when $d \ge 3$ as in the case $d=1$ (trivial) and $d=2$ (abelian). 

\subsection{Coordinates separating dilation, rotation and absolute shape}

We write $\{\e_{1}, \cdots, \e_{d}\}$ for a positive orthonomal basis of $\R^{d}$. 

For a non-zero vectors $\tilde{x} \in \R^{d}$ and a vector subspace $\pi \subset \R^{d}$ we denote by 
$OC(\tilde{x}, \pi)$
the orthogonal component of $\tilde{x}$ with respect to $\pi$. If $\{\tilde{\e}_{1}, \cdots, \tilde{\e}_{p}\}$ is a positive orthonormal basis of $\pi$, then 
$$OC(\tilde{x}, \pi)=\tilde{x}-\sum_{i=1}^{p} (\tilde{x} \cdot \tilde{\e}_{i}) \tilde{\e}_{i}.$$

We assume 
$$\hbox{span} \{q_{1}, \cdots, q_{N}\} \ge d-1,$$
so up to permutation of subindices we may assume
$$\hbox{span} \{q_{N-(d-2)}, \cdots q_{N}\} =d-1.$$

Under this assumption, there exists a unique rotation $R \in SO(d)$ such that 
\begin{itemize}
\item $R^{-1} q_{N} \in R_{+} \, \e_{1}$;
\item $R^{-1} \, OC(q_{N-1}, \hbox{span}\, q_{N}) \in R_{+} \, \e_{2}$;
\item $R^{-1} \, OC(q_{N-2}, \hbox{span}\, \{q_{N-1}, q_{N}\} ) \in R_{+} \, \e_{3}$; \\
$\cdots$
\item $R^{-1} \, OC(q_{N-(d-2)}, \hbox{span}\, \{q_{N-(d-1)}, q_{N}\} ) \in R_{+} \, \e_{d-1}$;
\end{itemize}

This way we write 
$$q=r R s,$$
in which $s \in \mathcal{Q}$ is such that 
$$\|s\|=1,$$ and
\begin{itemize}
\item $ s_{N}=\hat{b}_{1} \e_{1}$;
\item $OC(s_{N-1}, \e_{1}) \in \hat{b}_{2} \, \e_{2}$;
\item $OC(s_{N-2}, \hbox{span}\, \{ \e_{1}, \e_{2}\} ) \in \hat{b}_{3} \, \e_{3}$; \\
$\cdots$
\item $OC(s_{N-(d-2)}, \hbox{span}\, \{q_{N-(d-1)}, q_{N}\} ) \in \hat{b}_{d-1} \, \e_{d-1}$;
\end{itemize}
with $\hat{b}_{d-1}, \cdots \hat{b}_{1}>0$.

The velocity is then computed as 
$$\dot{q} = \dot{r} R s + r \dot{R} s + r R \dot{s} = R \left( \dot{r} s + r R^{-1} \dot{R} s + r \dot{s} \right).$$
or particle-wise
$$\dot{q}_{i} = \dot{r} R s_{i} + r \dot{R} s_{i} + r R \dot{s}_{i} = R \left( \dot{r} s_{i} + r R^{-1} \dot{R} s_{i} + r \dot{s}_{i} \right).$$

The Lagrangian is
\[
\begin{aligned}
	L &= \frac{1}{2} \langle \dot{q}, \dot{q} \rangle + U(q) \\
	&= \frac{1}{2} \dot{r}^2 + \frac{r^2}{2} \langle R^{-1} \dot{R} s+\dot{s}, R^{-1} \dot{R} s +\dot{s}\rangle  + \frac{U(s)}{r}\\
	&= \frac{1}{2} \dot{r}^2 + \frac{r^2}{2} \left[ \langle R^{-1} \dot{R} s, R^{-1} \dot{R} s \rangle + \langle \dot{s}, \dot{s} \rangle + 2 \langle R^{-1} \dot{R} s, \dot{s} \rangle \right] + \frac{U(s)}{r^{\kappa}}.
\end{aligned}
\]

For bivectors we have some useful expressions:
$$ a \wedge a=0, \qquad R a \wedge R b=R \,a \wedge b\, R^{-1}.$$
in which $R^{-1}=R^{*}$ is the adjoint of $R$. If we express $a \wedge b$ using matrix then $R^{*}$ is correspondingly represented by $R^{T}$.

The angular momentum bivector is then written in these coordinates as
\[
\begin{aligned}
C &=\sum m_{i} q_{i} \wedge \dot{q}_{i} \\
    &=\sum m_{i} (r R s_{i})  \wedge R ( \dot{r} s_{i} + r R^{-1} \dot{R} s_{i} + r \dot{s}_{i} ) \\
     &= r^{2} \sum m_{i} R \, (s_{i} \wedge R^{-1} \dot{R} s_{i} + s_{i} \wedge \dot{s}_{i} )\, R^{-1} . 
\end{aligned}
\]

The main problem we want to study here is the dynamics of a N-body system approaching a total collapse. We address the problem in the following setting:

\begin{itemize}
\item $d \ge 3$;
\item the force function of the mutual attraction of two particles in the initial system takes the form $U_{i, j}=m_{i} m_{j} \|q_{i}-q_{j}\|^{-\kappa}$ ($-\kappa-$homogeneous potential), and we suppose $0<\kappa<2$.
\end{itemize}

In this case,  \cite[Th\'eor\`eme Fondamental]{Chenciner} asserts that a total collapse is possible only if 

(a) the angular momentum bivector $C$ is zero, which is given by the relationship 

\begin{equation}\label{eq: zero-angular momentum condition}
\sum m_{i}  (s_{i} \wedge R^{-1} \dot{R} s_{i} + s_{i} \wedge \dot{s}_{i} )  =0, 
\end{equation}
and
(b) the limiting normalized configuration $s$ tends to a connected component of the set of (normalized) central configurations.

 A normalized central configuration is called isolated if it is isolated after the rotational symmetry on $\mathcal{N}$ is reduced.  The dimension of a configuration in $\R^{d}$ is the dimension of the vector subspace generated by this configuration. We shall work with the following assumption to have a unique limiting normalized configuration up to rotations:

\begin{assumption} The limiting central configuration is isolated, and is of dimension $d$ or $d-1$. 
\end{assumption}

\subsection{Non-abelian Lagrange-Routh Reduction at zero-moment level}
To proceed, we need to compute the Lagrange-Routh reduced system. This procedure has been carried out in \cite{Marsden-Ratiu-Scheurle} which we briefly sketch and apply to our particular case. 

Our symmetric group is the Lie group $G=SO(d)$, which is compact, with Lie algebra $\mathfrak{g}=so(d)$, which can be identified with its dual $so(d) \cong so(d)^{*}$ by means of a $SO(d)$-invariant metric (say the negative Killing form). An element $\xi \in \mathfrak{g}=so(d)$ can then be represented by a $d \times d$ anti-symmetric matrix 
$$(\xi_{ij})_{d \times d}, \xi_{ij}=-\xi_{ji},$$
 or equivalently by a bivector
 $$\sum_{i, j} \frac{1}{2} \xi_{{ji}} e_{i} \wedge e_{j}.$$

For $\xi \in \mathfrak{g}$, its associated fundamental vector field is denoted by $\underline{\xi}$. So $\underline{\xi}(q)$ is a vector at each $q \in \mathcal{Q}$. We write $\exp (t\xi)$ for the corresponding 1-parameter subgroup of $G$, which acts particle-wise as 
$$\exp (t \xi) \cdot q_{i} =\exp (t \xi) q_{i}$$
and thus the action of $\mathfrak{g}$ is given particle-wise as
$$\xi \cdot q_{i}=\xi q_{i}, $$
in which the right-hand side is given by matrix multiplication. 

By abuse of notation we have
$$\xi q_{i}= q_{i} \lrcorner \xi,$$
in which on the left-hand side $\xi$ represents an anti-symmetric matrix, and on the right-hand side $\xi$ represents the associated bivector. 

We may as well write the particle-wise action as
$$\xi \cdot q_{i}=q_{i} \lrcorner \xi. $$

So we have 
$$\underline{\xi}(q)=(\xi q_{1}, \cdots, \xi q_{N})=(q_{1} \lrcorner \xi, \cdots,q_{N} \lrcorner \xi) \in T_{q} \mathcal{Q}.$$

The moment map $\mu: TQ \mapsto \mathfrak{g}^{*}$ is define by the relationship 
$$\mu(v_{q}) \cdot \xi=\langle v_{q} , \underline{\xi}(q) \rangle.$$

In our case we may write $v_{q}=(q, v)$, so we obtain
$$\mu(v_{q}) \cdot \xi=\sum_{i} m_{i} \langle v_{i} , \xi q_{i} \rangle.$$
Using the identity
$$\langle v_{i} , \underline{\xi}(q_{i}) \rangle=\langle v_{i} , q_{i} \lrcorner \xi \rangle=\langle q_{i} \wedge v_{i}, \xi \rangle$$
and natural identification of $\mathfrak{g}$ with $\mathfrak{g}^{*}$ we get the angular momentum
$$\mu(v_{q})=\sum m_{i} q_{i} \wedge v_{i}= C. $$

A mechanical connection $\alpha: T\mathcal{Q} \mapsto \mathfrak{g}$ is an equivariant $\mathfrak{g}$-valued 1-form such that $\alpha(\underline{\xi}(q))=\xi$, and the horizontal distribution $\ker \alpha$ is orthogonal to the vertical distribution. 



Fix $c \in \mathfrak{g} \cong  \mathfrak{g}^{*}$. Then the corresponding reduced Routhian is given, as in \cite[III.B]{Marsden-Ratiu-Scheurle} by
$$R^{c} (v_{q})=L(v_{q})-\langle c, \alpha(v_{q}) \rangle,$$
in which $\mu(v_{q})=c$.

We are interested only in the case $c=0$, so the reduced Routhian is simply
$$R^{0} (v_{q})=\dfrac{1}{2} \|v_{q}\|^{2}-\mathcal{V}(q).$$

To write down the reduced equations of motion, we need in principle to study the contribution of the curvature of $\alpha$, defined as the $\mathfrak{g}$-valued 2-form $\beta$ such that 
$$\beta(u_{q}, v_{q})=d \alpha (Hor(u_{q}), Hor(v_{q})).$$
This has been extensively discussed in \cite[III. F, III. H]{Marsden-Ratiu-Scheurle}. In the case that interests us, we only need to observe that when $c=0$, the curvature 2-form brings no contribution to the reduced equations of motion.
 Part of the reduced equations of motion as given in \cite{Marsden-Ratiu-Scheurle}, namely (III. 37) is trivial. The only non-trivial equations, (III. 37) of \cite{Marsden-Ratiu-Scheurle} then take the form:

\[ \label{eq: reduced eq of motion}
\begin{cases}
	\dfrac{d}{dt} \dfrac{\partial R^0}{\partial \dot{r}} - \dfrac{\partial R^0}{\partial r} = 0 \\
	\dfrac{d}{dt} \dfrac{\partial R^0}{\partial \dot{s}} - \dfrac{\partial R^0}{\partial s} = 0.
\end{cases}
\]



\subsection{Precise computations of the reduced equations of motion}

We now carry out computations in our problem more precisely. 

We have set
$$q=r R s,$$
thus 
$$\dot{q} = \dot{r} R s + r \dot{R} s + r R \dot{s} = R \left( \dot{r} s + r R^{-1} \dot{R} s + r \dot{s} \right),$$ 
and thus the Lagrangian takes the form:
\[
\begin{aligned}
	L &= \frac{1}{2} \langle \dot{q}, \dot{q} \rangle + U(q) \\
	&= \frac{1}{2} \dot{r}^2 + \frac{r^2}{2} \langle R^{-1} \dot{R} s+\dot{s}, R^{-1} \dot{R} s +\dot{s}\rangle  + \frac{U(s)}{r}\\
	&= \frac{1}{2} \dot{r}^2 + \frac{r^2}{2} \left[ \langle R^{-1} \dot{R} s, R^{-1} \dot{R} s \rangle + \langle \dot{s}, \dot{s} \rangle + 2 \langle R^{-1} \dot{R} s, \dot{s} \rangle \right] + \frac{U(s)}{r}.
\end{aligned}
\]

To compute $R^{0}$, we determine the term $R^{-1} \dot{R}$ directly from the zero-angular momentum condition \eqref{eq: zero-angular momentum condition}:
$$
\sum m_{i}  (s_{i} \wedge R^{-1} \dot{R} s_{i} + s_{i} \wedge \dot{s}_{i} )  =0.
$$

This can be rewritten into the form
$$\mathcal{I}(s) \cdot R^{-1} \dot{R}+G(s, \dot{s})=0,$$
in which
$$\mathcal{I}(s) \cdot R^{-1} \dot{R}=\sum m_{i}  s_{i} \wedge R^{-1} \dot{R} s_{i},  \qquad G(s, \dot{s})=\sum m_{i}   s_{i} \wedge \dot{s}_{i}. $$

As $R^{-1} \dot{R}$ is anti-symmetric, it corresponds to a bivector 
$$\Gamma=\frac{1}{2} (R^{-1} \dot{R})_{{ji}} e_{i} \wedge e_{j}.$$
The inner product of a bivector is the natural one induced from the inner of vectors, namely 
$$\langle v_{1} \wedge v_{2}, v_{3} \wedge v_{4} \rangle=\langle v_{1}, v_{3}\rangle \langle v_{2}, v_{4}\rangle-\langle v_{1}, v_{4}\rangle \langle v_{2}, v_{3}\rangle.$$
We denote the induced norm for a bivector (and thus also for an antisymmetric matrix on $\R^{d}$) still by $\|\cdot\|$.

The product $R^{-1} \dot{R} s_{i}$ can be written in the equivalent form using contraction of a vector with a bivector as
$$R^{-1} \dot{R} s_{i}=s_{i} \lrcorner \Gamma.$$

So the moment of inertia tensor $\mathcal{I}(s)$ can be equivalently written as
$$\mathcal{I}(s) \Gamma=\sum m_{i} s_{i} \wedge (s_{i} \lrcorner \Gamma)=\sum m_{i} \Bigl(\|s\|^{2} \Gamma - s_{i} \lrcorner (s_{i} \wedge \Gamma) \Bigr). $$

\begin{lemma} $\mathcal{I}(s)$ is positive semi-definite. It is positive-definite when 
$$\hbox{span} \{s_{1}, s_{2}, \cdots, s_{N}\} \ge d-1.$$ 
\end{lemma}

{\bf Proof.} We compute the quadratic form
$$\langle \Gamma, \mathcal{I}(s) \Gamma \rangle=\langle \Gamma, \sum m_{i} s_{i} \wedge (s_{i} \lrcorner \Gamma) \rangle= \sum m_{i} \langle \Gamma, s_{i} \wedge (s_{i} \lrcorner \Gamma) \rangle=\sum m_{i} \langle s_{i} \lrcorner \Gamma,  s_{i} \lrcorner \Gamma \rangle \ge 0. $$
In the last step we have used the identity to transpose ``$s_{i} \wedge$'' to ``$s_{i} \lrcorner$''. 

The degeneracy can happen only if all the $s_{i}$ lie in the annihilator space of a non-zero bivector $\Gamma$. This is not possible if 
$$\hbox{span} \{s_{1}, s_{2}, \cdots, s_{N}\} \ge d-1.$$ 
Indeed the annihilator space of a non-zero bivector has dimension at most $d-2$. 
$\square$

We therefore make the following assumption:

\begin{assumption} The limiting normalized central configuration is isolated and is not contained in any vector-subspace of $\bar{\mathcal{Q}}$ of dimension $d-2$. 
\end{assumption}

With this assumption we have $\mathcal{I}(s)>0$ in a neighborhood of such a limiting normalized central configuration. If we do not distinguish an antisymmetric matrix with its corresponding 2-vector, we may write
\begin{equation}\label{eq: RIG}
R^{-1} \dot{R} = -\mathcal{I}(s)^{-1} G(s, \dot{s}).
\end{equation}

Thus

\[
\begin{aligned}
	R^{0} & =  \frac{1}{2} \dot{r}^2 + \frac{r^2}{2} \left[ \langle R^{-1} \dot{R} s, R^{-1} \dot{R} s \rangle + \langle \dot{s}, \dot{s} \rangle + 2 \langle R^{-1} \dot{R} s, \dot{s} \rangle \right] + \frac{U(s)}{r} \\
	&= \frac{1}{2} \dot{r}^2 + \frac{r^2}{2} \| R^{-1} \dot{R} s + \dot{s} \|^{2}  + \frac{U(s)}{r} \\
	&=\frac{1}{2} \dot{r}^2 + \frac{r^2}{2} \| \mathcal{I}(s)^{-1} G(s, \dot{s}) s - \dot{s} \|^{2}  + \frac{U(s)}{r}.
\end{aligned}
\]

We see that the term 
$$F(s, \dot{s}):= \| \mathcal{I}(s)^{-1} G(s, \dot{s}) s - \dot{s} \|^{2}=\dot{s}^{T} A(s) \dot{s} $$
is a quadratic form on $\dot{s}$. We show

\begin{lemma} The quadratic form $F(s, \dot{s})$  is positive-definite.
\end{lemma}

{\bf Proof.} 
This quadratic form is clearly non-negative. It is 0 only if 
$$ \mathcal{I}(s)^{-1} G(s, \dot{s}) s - \dot{s}=0.$$
Let $\Omega=\mathcal{I}(s)^{-1} G(s, \dot{s})(=-R^{-1} \dot{R})$, seen as an antisymmetric matrix. We have
$$ \Omega \,s_{i} - \dot{s}_{i}=0,\qquad \forall i=1, \cdots, N.$$
We have set 
$$s_{N}=\hat{b}_{1} \e_{1},  \quad s_{N-1}=\hat{c}_{21} \e_{1} + \hat{b}_{2} \e_{2}, \cdots, s_{N-d+1}=\hat{c}_{d-1, 1} \e_{1} + \cdots + \hat{b}_{d-1} \e_{d-1},$$
within which $\hat{b}_{1} > 0, \cdots, \hat{b}_{d=1} > 0$.

From the identity on $s_{N}$: 
$$ \Omega \,s_{N} - \dot{s}_{N}=0,$$
 we get
$$\hat{b}_{1} \Omega \e_{1}-\dot{\hat{b}}_{1} \e_{1}=0.$$
Since $\Omega$ is antisymmetric, with the Euclidean inner product in $\R^{d}$ we have $\langle \e_{1}, \Omega \e_{1}\rangle=0$. As $\hat{b}_{1} > 0$ we obtain
$$\Omega \e_{1}=0, \dot{\hat{b}}_{1}=0.$$

Next we consider the identity on $s_{N-1}$:
$$ \Omega \,s_{N-1} - \dot{s}_{N-1}=0,$$
which gives
$$\hat{c}_{21} \Omega \e_{1} + \hat{b}_{2} \Omega \e_{2}=\dot{\hat{c}}_{21} \e_{1} + \dot{\hat{b}}_{2} \e_{2}.$$
From the discussion above, $\Omega \e_{1}=0$, so we have
$$ \hat{b}_{2} \Omega \e_{2}=\dot{\hat{c}}_{21} \e_{1} + \dot{\hat{b}}_{2} \e_{2}.$$
Since $\e_{2}$ is orthogonal to both $\e_{1}$ and $\Omega \e_{2}$, we get 
$$ \dot{\hat{b}}_{2}=0, $$
thus
$$\Omega \e_{2}=k_{21} \e_{1}$$
for some $k_{21} \in \R$.

Consider the identity on $s_{N-2}$:
$$ \Omega \,s_{N-1} - \dot{s}_{N-1}=0.$$
We get
$$\hat{c}_{31} \Omega \e_{1}  + \hat{c}_{32} \Omega \e_{2}+ \hat{b}_{3} \Omega \e_{3}=\dot{\hat{c}}_{31} \e_{1} +\dot{\hat{c}}_{32} \e_{2} + \dot{\hat{b}}_{3} \e_{3}.$$
From the discussion above we have $ \Omega \e_{1} =0,  \Omega \e_{2}=k_{21} \e_{1}$, so we obtain
$$ \hat{c}_{32} k_{21} \e_{1}+ \hat{b}_{3} \Omega \e_{3}=\dot{\hat{c}}_{31} \e_{1} +\dot{\hat{c}}_{32} \e_{2} + \dot{\hat{b}}_{3} \e_{3}.$$
As $\e_{3}$ is orthogonal to $\e_{1}, \e_{2}, \Omega \e_{3}$, we obtain 
$$\dot{\hat{b}}_{3}=0$$
and thus 
$$\Omega \e_{3}=k_{31} \e_{1}+ k_{32} \e_{2}.$$

Continuing in this way, we see that in the basis $\{\e_{1}, \cdots, \e_{d}\}$ the antisymmetric matrix $\Omega$ has a zero (upper or lower) triangular part. Thus $\Omega$ is the zero matrix. This means
$G(s, \dot{s})$ is zero, thus $\dot{s}=0$.
$\square$

This way we have defined a norm on tangent vectors $\dot{s}$ at base point $s$ by 
$$\|\dot{s}\|_{FS}:=\sqrt{F(s, \dot{s})}.$$
The subscript ``FS'' is adopted from \cite{MM}, in which the norm corresponds to the Fubini-Study metric on $\mathbb{C} \mathbb{P}(N-2)$. 

We now write the reduced Routhian as
$$R^{0}  =\frac{1}{2} \dot{r}^2 + \frac{r^2}{2} F(s, \dot{s})  + \frac{U(s)}{r^{\kappa}}.$$

The reduced equations of motion 
\[
\begin{cases}
	\dfrac{d}{dt} \dfrac{\partial R^0}{\partial \dot{r}} - \dfrac{\partial R^0}{\partial r} = 0, \\
	\dfrac{d}{dt} \dfrac{\partial R^0}{\partial \dot{s}} - \dfrac{\partial R^0}{\partial s} = 0,
\end{cases}
\]
are thus
\[
\begin{cases}
	\ddot{r} = r \dot{s}^T A(s) \dot{s} -  \dfrac{ \kappa}{r^{\kappa+1}} U(s), \\
	\ddot{s} = \dfrac{1}{r^{\kappa+2}} A^{-1}(s) \nabla U(s) - 2 \dfrac{\dot{r} \dot{s}}{r} + \dfrac{1}{2} A^{-1}(s) \nabla (\dot{s}^{T} A(s) \dot{s})  - A^{-1}(s) \dot A(s) \dot{s}.
\end{cases}
\]

\section{The Reduced Equations of Motion and Finite Spin}\label{Sec: Finite Spin}
\subsection{Analysis of the reduced equations of motion}

To further analyze the dynamical behavior near a total collapse, we make a time-change by setting $dt = r^{\frac{\kappa}{2}+1} d\tau$. The equations of motion then take the form
\[
\begin{cases}
	r'' = (\dfrac{\kappa}{2}+1)\dfrac{ r'^2}{r} + r s'^T A(s) s' - \kappa\, r U(s), \\
	s'' = A^{-1}(s) \nabla U(s) + (\frac{\kappa}{2}-1)\dfrac{r' s'}{r} + \frac{1}{2} A^{-1}(s) \nabla (\dot{s}^{T} A(s) \dot{s})-  A^{-1}(s) \left( D A(s) \left( s' \right)\right) s'.
\end{cases}
\]

Further we set $r' = v r$, $s' = w$ to blow-up the set of total collisions $\{r=0\}$ \emph{\`a la }McGehee. This way we get

\[
\begin{cases}
	r' = v r, \\
	v' = \dfrac{\kappa}{2} v^2 + w^{T} A(s) w - \kappa U(s), \\
	s' = w, \\
	w' = A^{-1}(s) \nabla U(s) +(\frac{\kappa}{2}-1)v w + \dfrac{1}{2} A^{-1}(s) \nabla (w^{T} A(s) w) -  A^{-1}(s) \left( D A(s) \left( w \right)\right) w.
\end{cases}
\]

Setting 
\[
\widetilde{\nabla} = A^{-1}(s) \nabla,
\]
and
\[
\widetilde{D}_\tau w = w' - \frac{1}{2} \widetilde{\nabla} (w^{T} A(s) w) + A^{-1}(s) \left( D A(s) \left( w \right)\right) w,
\]
the system is now written in a more compact form as
\[
\begin{cases}
	r' = v r, \\
	v' = \dfrac{\kappa}{2} v^2 + w^{T} A(s) w - \kappa U(s), \\
	s' = w, \\
	\widetilde{D}_\tau w = \widetilde{\nabla} U(s) +(\frac{\kappa}{2}-1)v w .
\end{cases}
\]

The energy equation is 
$$\dfrac{1}{2} v^{2} + \dfrac{1}{2} F(s, w) -U(s)=r^{\kappa} h. $$

The collision manifold $\{r=0\}$ is invariant. Any equilibrium point on it can be expressed in these coordinates as $(0, v_{0}, s_{0}, 0)$, where from the energy equation $v_{0}^{2}=2 U(s_{0})$ and 
$$\widetilde{\nabla} U(s_{0})=0 \Leftrightarrow \nabla U(s_{0})=0 \Leftrightarrow \hbox{$s_{0}$ is a (reduced) normalized central configuration}.$$

Now let $q(t)$ be a total collision solution, with total collision happens at $t=t^{*}$, with normalized configuration converging to an isolated  normalized central configuration $s_{0}$. Then the corresponding solution $\gamma(\tau)=(r(\tau), v(\tau), s(t), w(\tau))$ converges as $\tau \to \infty$ to an isolated  equilibrium point $P(0, v_{0}, s_{0}, 0)$ with  
$$v(\tau) \to v_{0}=-\sqrt{2 U(s_{0})}<0.$$

From $r'=v r$ we see that $r(\tau)$ converges to 0 exponentially.

We want to show the arclength 
$$L(s)=\int \|s'(\tau)\|_{FS} \, d \tau=\int \|w(\tau)\|_{FS} \, d \tau. $$
is finite.

\begin{prop}\label{prop: MM}
	\begin{equation}\label{eq: arclength}
	L(s) = \int \|w(\tau)\|_{FS}\, d \tau
	\end{equation}
	is finite for a total collision solution.
\end{prop}

The proof is really that of Moeckel-Montgomery \cite{MM}, who established this result in the planar N-body problem with Newtonian potential. In the following sketch we explain why this holds in our case and indicate some rather minor differences. 

{\bf Sketch of Proof. } As compare to the argument of \cite{MM}, the dimension is increased and we need to carefully put coefficients depending on $\kappa$ in various places. In the $(\delta r, \delta v)$ directions the block is now $\begin{pmatrix} v_{0} & 0 \\ 0 & \kappa v_{0} \end{pmatrix}$ giving the rate of exponential convergence of $r(\gamma) \to 0$, complemented by the block  
$$B=\begin{pmatrix} 0 & Id \\ D \widetilde{\nabla} U(s_{0}) & (\kappa/2-1) v_{0} Id  \end{pmatrix}$$ 
in the $(\delta s, \delta w)$ directions. 

We look for eigenvalues and eigenvectors of the matrix $B$. If $(\delta s, \delta w)$ is an eigenvector with eigenvalue $\lambda$, then we get 
$$\delta w=\lambda \delta s, \qquad D \widetilde{\nabla} U(s_{0}) \delta s+(\kappa/2-1) v_{0} \delta w=\lambda \delta w $$
which is reduced to the system
$$\delta w= \lambda \delta s, \qquad D \widetilde{\nabla} U(s_{0}) \delta s=(\lambda-(\kappa/2-1) v_{0}) \lambda \delta s=:c\, \delta s.$$ 
So $\delta s$ is an eigenvector of the central configuration $s_{0}$ with eigenvalue $c$, {just with the gradient $\nabla$ replaced by $\widetilde{\nabla}=A^{-1}(s) \nabla$.}

 Solving $\lambda$ from $c$ we get
$$\lambda_{\pm}=\dfrac{(\kappa-2) v_{0} \pm \sqrt{(\kappa-2)^{2} v_{0}^{2} +16 c }}{4}.$$

Since $0<\kappa<2, v_{0}<0$, nonreal eigenvalues have positive real parts, corresponding to unstable directions. Also, we have $\lambda_{-}=0$ if and only if $c=0$.

So if the matrix $D \nabla U(s_{0})$ is non-singular, the corresponding point $P$ is hyperbolic and any solution approaching $P$ lies in its stable manifold, thus with exponential rate of convergence. Thus the integrand in \eqref{eq: arclength} is finite.

In the case when $D \nabla U(s_{0})$ is singular, we may proceed as in \cite{MM} using center manifold theorem \cite{CenterManifold}. Besides the change of dimension which does not affect anywhere in the argument, the precise formulas slightly differ from those in  \cite{MM} by coefficients depending on $\kappa$, which thanks to our assumption $0 < \kappa < 2$ do not change the sign nor the order estimates of the formulas.  Consequently, \cite[Lemma 4.1]{MM} still holds. The argument leading to \cite[Lemma 4.2]{MM} is completely formal and works also in our situation. Also \cite[Lemma 4.3]{MM} holds by the very same argument as in \cite{MM}, just now we need to write $k=-\frac{2}{\kappa-2} v(0)^{-1}$. Nothing changes for the arguments leading to \cite[Lemma 4.4]{MM} and to  \cite[Lemma 4.5]{MM}. Consequently \cite[Theorem 4.6]{MM} holds in our situation. This is the statement of the Proposition.
$\square$

\subsection{Finite spin}

Let $s_{0}$ be an isolated normalized reduced central configuration. It then represents  an $SO(d)$-invariant family $Orb(s_{0})$ of normalized central configurations  in $\mathcal{N}$. Then there exists an $SO(d)$-invariant neighborhood $\mathcal{U}(s_{0})$ of $Orb(s_{0})$ in  $\mathcal{N}$ such that along any normalized total collision orbit with limiting normalized reduced central configuration $s_{0}$ and with total collision time $t^{*}$, there exists $t_{-}$, such that the solution in the time interval $(t_{-}, t^{*})$ lies entirely in  $\mathcal{U}(s_{0})$.

\begin{lemma} \label{lem: estimate}
In the time interval $ (t_{-}, t^{*})$, then there exists a positive number $C_{1}>0$ depending only on $\mathcal{U}(s_{0})$, such that
	\[
	\| R^{-1} R' \| \leq C_{1} \|w\|_{FS}. 
	\]
\end{lemma}

{\bf Proof.}
Treating $\mathcal{I}(s), A(s)$ as positive-definite symmetric operators we have at any isolated normalized central configuration $s_{0}$ with dimension $\ge d-1$, there hold
$$2 c_{1}:=\lambda_{\mathcal{I}}(s_{0})=\min_{\|\Gamma\|=1} \langle \Gamma, \mathcal{I}(s_{0}) \Gamma \rangle>0
 \qquad 2 c_{2}:=\lambda_{A} (s_{0})=\min_{\|w\|=1} \langle w, A(s_{0}) w\rangle>0
.$$
Shrinking $\mathcal{U}(s_{0})$ if necessary we may assume for all $s \in \mathcal{U}(s_{0})$ we have
$$\lambda_{\mathcal{I}}(s)=\min_{\|\Gamma\|=1} \langle \Gamma, \mathcal{I}(s) \Gamma \rangle > c_{1}>0
 \qquad \lambda_{A}(s)=\lambda_{A} (s)=\min_{\|w\|=1} \langle w, A(s_{0}) w\rangle>c_{2}>0
.$$
Changing time in Eq. \eqref{eq: RIG}, we have
$$R^{-1} R'=\mathcal{I}(s)^{-1} G(s, s')=\mathcal{I}(s)^{-1} G(s, w).$$
Then
$$\|R^{-1} R'\| < c_{1}^{-1} \|\mathcal{I} R^{-1} R' \|=c_{1}^{-1} \|G(s, w)\|.$$

Now 
$$\|G(s, w)\| \le \sum_{i} m_{i} \|s_{i}\| \|w_{i}\| \le \sqrt{\sum_{i} m_{i} \|s_{i}\|^{2}}  \sqrt{\sum_{i} m_{i} \|w_{i}\|^{2}}.$$
As we are on $\mathcal{N}$, we have
$$ \sqrt{\sum_{i} m_{i} \|s_{i}\|^{2}} =\|s\|=1,$$
so we get
$$\|G(s, w)\| \le \|w\| \le c_{3}(s) \|w\|_{FS} \le c_{4}  \|w\|_{FS},$$
in which the second inequality is based on equivalence of norms on finite-dimensional vector spaces. Moreover $c_{3}(s)$ can be controlled by $c_{2}$ and the latter continuous. Further shrinking $\mathcal{U}(s_{0})$ if necessary, we get the last inequality. Combining all these we get the assertion.
$\square$

\medskip

With this we prove Theorem \ref{thm: main}:
\medskip

{\bf Proof of Theorem \ref{thm: main}}

By equivalence of norms on a finite-dimensional vector space, there exists a constant $C_{2}>0$ such that 
$$\| R R^{-1} R'  \| \leq C_{2} \| R \| \| R^{-1} R' \|$$
holds.

By Lemma \eqref{lem: estimate}
\[
\| R' \| =  \| R R^{-1} R'  \| \leq C_{2} \| R \| \| R^{-1} R' \| \leq C_{2} C_{1}\,  \|w\|_{FS}. 
\]

The theorem now follows from Proposition \ref{prop: MM}.
$\square$



\medskip
\medskip

\noindent
Xiang Yu, \\
Center for Applied Mathematics, Tianjin University, PR China.\\
Email: xiang.zhiy@foxmail.com;\\

\noindent
Lei Zhao, \\
School of Mathematical Sciences, Dalian University of Technology, PR China.\\
Email: zhao1899@dlut.edu.cn.

\end{document}